\documentclass[english]{extarticle}
\usepackage[T1]{fontenc}
\usepackage[latin9]{inputenc}
\usepackage{verbatim}
\usepackage{amsmath}
\usepackage{amsthm}
\usepackage{amssymb}

\makeatletter
\numberwithin{equation}{section}
\numberwithin{figure}{section}
\theoremstyle{plain}
\newtheorem{thm}{\protect\theoremname}
  \theoremstyle{remark}
  \newtheorem{rem}[thm]{\protect\remarkname}
  \theoremstyle{plain}
  \newtheorem{cor}[thm]{\protect\corollaryname}
  \theoremstyle{definition}
  \newtheorem{example}[thm]{\protect\examplename}
  \theoremstyle{definition}
  \newtheorem{defn}[thm]{\protect\definitionname}

\author{Lin Jiu \\ 
Research Institute for Symbolic Computation\\
Johannes Kepler University \\
4040 Linz, Austria\\
ljiu@risc.uni-linz.ac.at
}

\date{}

\makeatother

\usepackage{babel}
  \providecommand{\corollaryname}{Corollary}
  \providecommand{\definitionname}{Definition}
  \providecommand{\examplename}{Example}
  \providecommand{\remarkname}{Remark}
\providecommand{\theoremname}{Theorem}

\begin{document}

\title{Integral Representations of Equally Positive Integer-Indexed Harmonic
Sums at Infinity\thanks{This work is supported by the Austrian Science Fund (FWF) grant SFB
F50 (F5006-N15 and F5009-N15)}}
\maketitle
\begin{abstract}
We identify a partition-theoretic generalization of Riemann zeta function
and the equally positive integer-indexed harmonic sums at infinity,
to obtain the generating function and the integral representations
of the latter. The special cases coincide with zeta values at positive
integer arguments.
\end{abstract}

\section{\label{sec:Introduction}Introduction}

The harmonic sum of \emph{indices} $a_{1},\ldots,a_{k}\in\mathbb{R}\backslash\left\{ 0\right\} $
is defined as (see \cite[eq.~4, pp.~1]{HarmonicSSumDEF})
\[
S_{a_{1},\ldots,a_{k}}\left(N\right)=\sum_{N\geq n_{1}\geq\cdots\ge n_{k}\ge1}\frac{\text{sign}\left(a_{1}\right)^{n_{1}}}{n_{1}^{\left|a_{1}\right|}}\times\cdots\times\frac{\text{sign}\left(a_{k}\right)^{n_{k}}}{n_{k}^{\left|a_{k}\right|}},
\]
which is naturally connected to the Riemann zeta function, by noting
that $N=\infty$, $k=1$ and $a_{1}>0$ gives $S_{a_{1}}\left(\infty\right)=\zeta\left(a_{1}\right)$.
A variety of the study can be found in the literature. For instance,
Hoffman \cite{Hoffman} established the connection between harmonic
sums and multiple zeta values. We especially focus on the\emph{ equally
positively indexed harmonic sums}, given by the case $a_{1}=\cdots=a_{k}=a>0$
\begin{equation}
S_{\boldsymbol{a}_{k}}\left(N\right):=S_{\underset{k}{\underbrace{a,\ldots,a}}}\left(N\right)=\sum_{N\geq n_{1}\geq\cdots\ge n_{k}\ge1}\frac{1}{\left(n_{1}\cdots n_{k}\right)^{a}},\label{eq:HarmonicSSum}
\end{equation}
and also the \emph{equally positive integer-indexed harmonic sums
}(EPIIHS), namely $a=m\in\mathbb{Z}_{>0}$. If $N=\infty$, we additionally
assume $m\in\mathbb{Z}_{>1}$ for convergence.

Recently, Schneider \cite{PartitionZeta} studied the generalized
$q$-Pochhammer symbol and obtained \cite[pp.~3]{PartitionZeta}
\begin{equation}
\prod_{n\in X}\frac{1}{1-f\left(n\right)q^{n}}=\sum_{\lambda\in\mathcal{P}_{X}}q^{\left|\lambda\right|}\prod_{\lambda_{i}\in\lambda}f\left(\lambda_{i}\right),\label{eq:subPartitionZeta}
\end{equation}
where
\begin{itemize}
\item $X\subseteq\mathbb{\mathbb{Z}}_{>0}$ and $f:\mathbb{\mathbb{Z}}_{>0}\longrightarrow\mathbb{C}$
such that if $n\not\in X$ then $f\left(n\right)=0$;
\item $\mathcal{P}_{X}$ is the set of partitions into elements of $X$;
\item $\lambda\vdash n$ means $\lambda$ is a partition of $n$, the size
$|\lambda|$ is the sum of the parts of $\lambda$, i.e., the number
$n$ being partitioned, and $\lambda_{i}\in\lambda$ means $\lambda_{i}\in\mathbb{Z}_{>0}$
is a part of partition $\lambda$. 
\end{itemize}
Further define $l\left(\lambda\right):=k$, $n_{\lambda}:=\lambda_{1}\cdots\lambda_{k}$
and denote $\mathcal{P}:=\mathcal{P}_{\mathbb{Z}_{>0}}$. Noting $\lambda_{1}\geq\cdots\geq\lambda_{k}\geq1$,
a \emph{partition-theoretic generalization of Riemann zeta function
}\cite[eq.~11, pp.~4]{PartitionZeta} is defined and identified as
\begin{equation}
\zeta_{\mathcal{P}}\left(\left\{ a\right\} ^{k}\right):=\sum_{l\left(\lambda\right)=k}\frac{1}{n_{\lambda}^{a}}=\sum_{\lambda_{1}\geq\cdots\geq\lambda_{k}\geq1}\frac{1}{\lambda_{1}^{a}\cdots\lambda_{k}^{a}}=S_{\boldsymbol{a}_{k}}\left(\infty\right),\label{eq:partition_zeta_function}
\end{equation}
which leads to the generating function and the integral representation
of $S_{\boldsymbol{m}_{k}}\left(\infty\right)$, presented in the
next section.

\section{Main results}

We first apply (\ref{eq:subPartitionZeta}) to the case $X=\left\{ 1,2,\ldots,N\right\} $
and $f\left(n\right):=\frac{t^{a}}{n^{a}}$, obtaining
\[
\prod_{n=1}^{N}\frac{1}{1-\frac{t^{a}}{n^{a}}q^{n}}=\sum_{\lambda\in\mathcal{P}_{X}}q^{\left|\lambda\right|}\prod_{\lambda_{i}\in\lambda}\frac{t^{a}}{\lambda_{i}^{a}}=\sum_{\lambda\in\mathcal{P}_{X}}q^{\left|\lambda\right|}\frac{t^{l\left(\lambda\right)a}}{n_{\lambda}^{a}},
\]
which, by further letting $q\rightarrow1$, yields the following generating
function.
\begin{thm}
The generating function of $S_{\boldsymbol{a}_{k}}\left(N\right)$
is given by
\begin{equation}
\sum_{k=0}^{\infty}S_{\boldsymbol{a}_{k}}\left(N\right)t^{ak}=\prod_{n=1}^{N}\frac{n^{a}}{n^{a}-t^{a}}.\label{eq:SpecialApplicationSubsetPartition}
\end{equation}
\end{thm}
\begin{rem}
The special case for $a=1$ is \cite[eq.~9, pp.~1272]{BetaTaylorExpansion}
\begin{equation}
\sum_{k=0}^{\infty}t^{k}S_{\boldsymbol{1}_{k}}\left(N\right)=\frac{N!}{\left(1-t\right)\cdots\left(N-t\right)}=N\cdot B\left(N,1-t\right),\label{eq:Bluemlein}
\end{equation}
involving the beta function $B$, defined by
\begin{equation}
B\left(x,y\right):=\int_{0}^{1}z^{x-1}\left(1-z\right)^{y-1}dz=\frac{\Gamma\left(x\right)\Gamma\left(y\right)}{\Gamma\left(x+y\right)},\label{eq:BetaIntegralRepresentation}
\end{equation}
where the integral representation holds for $\text{Re}\left(x\right)$,
$\text{Re}\left(y\right)>0$. 
\end{rem}
\begin{cor}
For $m\in\mathbb{Z}_{>1}$, denote $\xi_{m}:=\exp\left(\frac{2\pi\text{i}}{m}\right)$
with $\text{i}^{2}=-1$. Then, 
\begin{equation}
\sum_{k=0}^{\infty}S_{\boldsymbol{m}_{k}}\left(\infty\right)t^{mk}=\prod_{j=0}^{m-1}\Gamma\left(1-\xi_{m}^{j}t\right).\label{eq:generatingFunction}
\end{equation}
\end{cor}
\begin{proof}
From (\ref{eq:SpecialApplicationSubsetPartition}) and (\ref{eq:Bluemlein}),
we have
\[
\sum_{k=0}^{\infty}S_{\boldsymbol{m}_{k}}\left(N\right)t^{mk}=\prod_{n=1}^{N}\frac{n^{m}}{\left(n-\xi_{m}^{0}t\right)\cdots\left(n-\xi_{m}^{m-1}t\right)}=\prod_{j=0}^{m-1}N\cdot B\left(N,1-\xi_{m}^{j}t\right).
\]
Then, apply the limit (see \cite[pp.~254, ex.~5]{ModernAnalysis})
$\Gamma\left(z\right)=\underset{n\rightarrow\infty}{\lim}N^{z}B\left(N,z\right)$
to $z_{j}=1-\xi_{m}^{j}t$, $j=0,\ldots,m-1$, by noting $\xi_{m}^{0}+\cdots+\xi_{m}^{m-1}=0$,
to complete the proof.
\end{proof}
\begin{rem}
An alternative proof can be given by letting $N=\infty$ in (\ref{eq:SpecialApplicationSubsetPartition})
and applying \cite[Thm.~1.1, pp.~547]{Armin}.
\end{rem}
\begin{rem}
For general $a>0$, we failed to obtain a closed form of $\overset{\infty}{\underset{n=1}{\prod}}\frac{n^{a}}{n^{a}-t^{a}}.$
\end{rem}
\begin{example}
\label{exa:mequals2} When $m=2$, we apply (\ref{eq:generatingFunction})
to get
\[
B\left(1+t,1-t\right)=\Gamma\left(1+t\right)\Gamma\left(1-t\right)=\sum_{k=0}^{\infty}S_{\boldsymbol{2}_{k}}\left(\infty\right)t^{2k}.
\]
From the integral representation (\ref{eq:BetaIntegralRepresentation}),
we obtain (also see Remark \ref{rem:Interchange})
\begin{equation}
B\left(1+t,1-t\right)=\int_{0}^{1}z^{t}\left(1-z\right)^{-t}dz=\sum_{k=0}^{\infty}\frac{t^{k}}{k!}\int_{0}^{1}\log^{k}\left(\frac{z}{1-z}\right)dz.\label{eq:Interchange}
\end{equation}
Then it follows, by comparing coefficients of $t$, 
\[
S_{\boldsymbol{2}_{k}}\left(\infty\right)=\frac{1}{\left(2k\right)!}\int_{0}^{1}\log^{2k}\left(\frac{z}{1-z}\right)dz.
\]
In particular, $k=1$ yields
\[
\frac{\pi^{2}}{6}=\zeta\left(2\right)=S_{2}\left(\infty\right)=\frac{1}{2}\int_{0}^{1}\log^{2}\left(\frac{z}{1-z}\right)dz.
\]
\end{example}
\begin{rem}
\label{rem:Interchange}We may interchange the integral and the sum
of the series in (\ref{eq:Interchange}), by restricting $t$ to a
closed compact set, e.g., $\left[-\frac{1}{2},\frac{1}{2}\right]$,
satisfying $\text{Re}\left(1-t\right)$, $\text{Re}\left(1+t\right)>0$
as that in (\ref{eq:BetaIntegralRepresentation}), in order to guarantee
uniform convergence of the integral representation. (Similar discussion
is omitted for the multiple beta function, defined next.)
\end{rem}
\begin{defn}
The \emph{multiple beta function }\cite[Ch.~49]{DirichletDistribution}
is defined as
\begin{equation}
B\left(\alpha_{1},\ldots,\alpha_{m}\right):=\frac{\Gamma\left(\alpha_{1}\right)\cdots\Gamma\left(\alpha_{m}\right)}{\Gamma\left(\alpha_{1}+\cdots+\alpha_{m}\right)}=\int_{\Omega_{m}}\prod_{i=1}^{m}x_{i}^{\alpha_{i}-1}d\mathbf{x},\label{eq:MultiBetaIntegralRepresentation}
\end{equation}
where $\Omega_{m}=\left\{ \left(x_{1},\ldots,x_{m}\right)\in\mathbb{R}_{>0}^{m}:x_{1}+\cdots+x_{m-1}<1,\ x_{1}+\cdots+x_{m}=1\right\} $
and the integral representation requires $\text{Re}\left(\alpha_{1}\right),\ldots,\text{Re}\left(\alpha_{m}\right)>0$.
\end{defn}
Following the same idea as that in Example \ref{exa:mequals2}, we
first have, from (\ref{eq:generatingFunction}),
\[
B\left(1-\xi_{m}^{0}t,\ldots,1-\xi_{m}^{m-1}t\right)=\frac{1}{\left(m-1\right)!}\sum_{k=0}^{\infty}S_{\boldsymbol{m}_{k}}\left(\infty\right)t^{mk}.
\]
Then, apply the integral representation (\ref{eq:MultiBetaIntegralRepresentation}),
expand the integrand as a power series in $t$, and compare coefficients
of $t$, to obtain the following integral representation.
\begin{thm}
For all $m,\,k\in\mathbb{Z}_{>0}$ with $m\geq2$,
\[
S_{\boldsymbol{m}_{k}}\left(\infty\right)=\frac{\left(-1\right)^{mk}\left(m-1\right)!}{\left(mk\right)!}\int_{\Omega_{m}}\log^{mk}\left(\prod_{j=0}^{m-1}x_{j+1}^{\xi_{m}^{j}}\right)d\mathbf{x}.
\]
\end{thm}
\begin{cor}
In particular, the case $k=1$ implies for integer $m\in\mathbb{Z}_{>1}$
that
\[
\zeta\left(m\right)=\frac{\left(-1\right)^{m}}{m}\int_{\Omega_{m}}\log^{m}\left(\prod_{j=0}^{m-1}x_{j+1}^{\xi_{m}^{j}}\right)d\mathbf{x},
\]
or alternatively
\begin{eqnarray*}
\zeta\left(m\right) & = & \frac{\left(-1\right)^{m}}{m}\int_{0}^{1}\int_{0}^{1-x_{1}}\cdots\int_{0}^{1-x_{1}-\cdots-x_{m-2}}\\
 &  & \log^{m}\left(x_{1}^{\xi_{m}^{0}}\cdots x_{m-1}^{\xi_{m}^{m-2}}\left(1-x_{1}-\cdots-x_{m-1}\right)^{\xi_{m}^{m-1}}\right)dx_{m-1}\cdots dx_{1}.
\end{eqnarray*}
\end{cor}

\section{Acknowledgment}

The author would like to thank Dr.~Jakob Ablinger for his help on
harmonic sums; Prof.~Johannes Bl\"{u}mlein for his handwritten notes
on the proof of (\ref{eq:Bluemlein}); and especially his mentors,
Prof.~Peter Paule and Prof.~Carsten Schneider, for their valuable
suggestions.

\end{document}